\documentclass[11pt]{amsart}
\usepackage[T1]{fontenc}
\usepackage[latin9]{inputenc}
\usepackage{amsmath}
\usepackage{amsthm}
\usepackage{amssymb}
\usepackage{esint}
\usepackage{xcolor}

\makeatletter
\theoremstyle{plain}
\newtheorem{thm}{\protect\theoremname}
\theoremstyle{definition}
\newtheorem{defn}[thm]{\protect\definitionname}
\theoremstyle{plain}
\newtheorem{prop}[thm]{\protect\propositionname}
\theoremstyle{plain}
\newtheorem*{thm*}{\protect\theoremname}
\theoremstyle{plain}
\newtheorem*{cor*}{\protect\corollaryname}

\newcommand{\fH}{\mathcal{H}}
\newcommand{\fK}{\mathcal{K}}
\newcommand{\fM}{\mathcal{M}}

\newcommand{\fL}{\mathcal{L}}
\newcommand{\myomega}{t}

\usepackage[english]{babel}
\providecommand{\definitionname}{Definition}
\providecommand{\propositionname}{Proposition}
\providecommand{\theoremname}{Theorem}
\providecommand{\corollaryname}{Corollary}

\title[POVMs and Densely-Defined OVFs]{Positive Operator-Valued Measures and Densely-Defined Operator-Valued Frames}

\author{Benjamin Robinson}
     \address{Benjamin Robinson, US Air Force Research Lab, 2241 Avionics Cir. \\ WPAFB, OH 45433, USA}
     \email{benjamin.robinson.8@us.af.mil}
     \thanks{This work was funded in part by the U.S. Air Force Office of Scientific Research under Grant No. FA9550-12-1-0418 and Grant No. FA9550-12-1-0225.}

     \author{Bill Moran}
     \address{Department of Electrical \& Electronic Engineering, The University of Melbourne, Parkville, 3010, Australia}
     \email{wmoran@unimelb.edu.au}
     
     \author{Doug Cochran}
     \address{Douglas Cochran, School of Mathematical and Statistical Sciences, Arizona State University, Tempe, AZ 85287-1804 USA}
     \email{cochran@asu.edu}

     \date{April 7, 2020}


\begin{document}
     \begin{abstract}
In the signal-processing literature, a frame is a mechanism for performing analysis and reconstruction in a Hilbert space.  By contrast, in quantum theory, a positive operator-valued measure (POVM) decomposes a Hilbert-space vector for the purpose of computing measurement probabilities. 
     Frames and their most common generalizations can be seen to give rise to POVMs, but does every reasonable POVM arise from a type of frame?  In this paper we answer this question using a Radon-Nikodym-type result.
     \end{abstract}
     \maketitle

\section{Introduction}

Originally studied in quantum mechanics, positive operator-valued measures (POVMs) have recently been suggested in the signal-processing literature as a natural general framework for the analysis and reconstruction of square-integrable functions \cite{moran2013positive,han2014operator,robinson2014operator,li2017radon}.  Examples include frames \cite{duffin1952class,daubechies1986painless,christensen2002introduction}, g-frames \cite{sun2006g,kaftal2009operator,han2011operator}, continuous frames \cite{ali1993continuous}, and the overarching continuous g-frames \cite{abdollahpour2008continuous}, all of which we will call simply operator-valued frames (OVFs).  The precise relationship between OVFs and POVMs is captured by a Radon-Nikodym-type theorem \cite[Theorem~3.3.2]{robinson2014operator}.  But this theorem applies only to POVMs with sigma-finite total variation.  The purpose of this paper is to show that the above relationship extends to non-sigma-finite POVMs and densely-defined OVFs.

Even the sigma-finite case has been attempted by many authors. In finite dimensions, for example, Chiribella et al. recently solved the problem using barycentric decompositions \cite{chiribella2010barycentric}.  In infinite dimensions, extra assumptions have traditionally been used, such as the finiteness of the POVM's trace \cite{berezanskiui1995spectral} or the weak compactness of its range \cite[VI.8.10]{dunford1967linear}.
One of the main complications in this case is that the space of bounded linear operators on an infinite-dimensional Hilbert space does not have the so-called Radon-Nikodym property \cite{barcenas2003radon}.  Another reason the problem is challenging is that it requires proving the conclusion of Naimark's Theorem \cite{neumark1940spectral,neumark1943representation}.  The non-sigma-finite case is at least as hard, and as far as we know no one has yet attempted it.

The paper is organized as follows.  In Section~\ref{sec:the-question}, we make the relationship we wish to establish precise.  Then, in Section~\ref{sec:main-result} we prove a Radon-Nikodym-type theorem from which the desired relationship follows.

\section{The Question} \label{sec:the-question}

Let $\fH$ be a separable complex Hilbert space throughout, whose inner product is conjugate-linear in its second variable.  A \emph{g-frame} for $\fH$ is
a sequence of bounded operators $T_{1},T_{2},\dots$ mapping $\fH$
into the sequence of separable Hilbert spaces $\fK_{1},\fK_{2},\dots$
such that the operator $T:x\in\fH\mapsto\left\{ T_{i}x\right\} _{i\ge1}$
maps into $\bigoplus_{i\ge1}\fK_{i}$ and is bounded above and below
\cite{sun2006g}. A related concept is the concept of a \emph{frame}
for $\fH$. This is usually defined as follows.
\begin{defn}
A sequence $x_{1},x_{2},\dots$
in $\fH$ is said to be a \emph{frame }for $\fH$ if the operator
$T:x\in\fH\mapsto\left\{ \left\langle x,x_{i}\right\rangle \right\} _{i\ge1}$
maps into $\ell^{2}(\mathbb{N})$ and is bounded above and below,
or equivalently, if there are $A,B>0$ such that
\[
A\left\Vert x\right\Vert _{\fH}^{2}\le\sum_{i\ge1}\left|\left\langle x,x_{i}\right\rangle \right|^{2}\le B\left\Vert x\right\Vert _{\fH}^{2}
\]
for all $x\in\fH$. The numbers $A$ and $B$ are called the \emph{frame bounds}.
\end{defn}
Using the self-duality of Hilbert spaces, we can easily see that every
frame $\{x_{i}\}_{i\ge1}$ for $\fH$ can be uniquely expressed as
the g-frame $\{T_{i}\}_{i\ge1}$ for $\fH$ with $T_{i}=\left\langle \cdot,x_{i}\right\rangle $,
so that the concept of a g-frame subsumes the concept of a frame.
For examples of frames of interest, see \cite{duffin1952class,daubechies1986painless,daubechies1992ten,christensen2002introduction}.
For examples of g-frames of interest that are not frames, see Examples
2.3.7 and 2.3.8 in \cite{robinson2014operator}. 

The importance of frames and g-frames is that they provide a framework
for \emph{analysis, synthesis, }and \emph{reconstruction} of a function.
The operator $T$ above is usually referred to as the \emph{analysis
operator}. \emph{Reconstruction} can be performed if we can calculate
the result of applying $S^{-1}T^{*}$ to $Tx$, where $T^{*}$ is
the adjoint of $T$ and $S=T^{*}T$ is known as the \emph{frame operator}. The intermediate step of applying
$T^{*}$ to $Tx$ is called \emph{synthesis. }In practice, if one
knows only $T$ and $Tx$, one way to recover $x$ without directly inverting
$S$ is by way of the so-called \emph{frame algorithm}:
\begin{prop}
\cite{grochenig1993acceleration} Let $\{x_{i}\}_{i\ge1}$ be a frame
for $\fH$ with frame bounds $A,B>0$. Then every $x\in\fH$ can be
reconstructed from $T$ and $Tx$ alone using the iteration
\[
x^{(n)}=x^{(n-1)}+\frac{2}{A+B}S\left(x-x^{(n-1)}\right),
\]
for $n\ge 1$ with $x^{(0)}=0$. Further, $x^{(n)}\to x$ according to
\[
\left\Vert x-x^{(n)}\right\Vert \le\left(\frac{B-A}{B+A}\right)^{n}\left\Vert x\right\Vert .
\]
\end{prop}
\noindent This algorithm applies equally well to frames and g-frames.

Another paradigm for analysis and synthesis is that of \emph{continuous
frames} \cite{ali1993continuous}, which we now describe.
\begin{defn}
Let $(\Omega,\Sigma,\mu)$ be a measure space and let $\{x_{\myomega}\}_{\myomega\in\Omega}\subset\fH$
be such that for all $x\in\fH$, the map $\myomega\in\Omega\mapsto\left\langle x,x_{\myomega}\right\rangle $
is measurable. Then $\left(\mu,\{x_t\}_{\myomega\in\Omega}\right)$ is a \emph{continuous frame} if $T:x\in\fH\mapsto\left\{ \left\langle x,x_{\myomega}\right\rangle \right\} _{\myomega\in\Omega}$
maps into $L^{2}(\mu)$ and is bounded above and below. This boundedness condition can also be expressed by saying there are constants $A,B>0$ such that
\[
A\left\Vert x\right\Vert _{\fH}^{2}\le\int_{\Omega}\left|\left\langle x,x_{\myomega}\right\rangle \right|^{2}\,d\mu(\myomega)\le B\left\Vert x\right\Vert _{\fH}^{2}
\]
for all $x\in\fH$.
\end{defn}
\noindent Examples of interest occur
in wavelet and Gabor analysis \cite{daubechies1986painless,daubechies1992ten}. 

Encompassing both continuous frames and g-frames is the overarching concept of
a \emph{continuous g-frame} \cite{abdollahpour2008continuous}.
We will generalize this concept slightly and call the result an operator-valued frame.  (We note that our terminology differs from the term ``operator-valued frame'' in \cite{kaftal2009operator}, which refers to what we call here g-frames.)
\begin{defn}
Let $(\Omega,\Sigma,\mu)$ be a measure space. Let $\fM$ be a dense linear
subspace of $\fH$. Let $\{\fK(\myomega)\}_{\myomega\in\Omega}$ be a $\Sigma$-measurable
family of Hilbert spaces\footnote{See \cite{dixmier2011neumann} or \cite{robinson2014operator} for the definition of this and  the direct integral $\int^\oplus_\Omega \fK(\myomega)\, d\mu(\myomega)$ of these spaces.}
and $T(\myomega):\fM\to\fK(\myomega)$ for every $\myomega\in\Omega$ be a
family of linear maps such that for every $x\in \fM$ both $\{T(\myomega)x\}_{\myomega\in\Omega}\in\int_{\Omega}^{\oplus}\fK(\myomega)\,d\mu(\myomega)$
and the map $T:x\in \fM\mapsto\{T(\myomega)x\}_{\myomega\in\Omega}$ is
bounded above and below. Then we say $\left(\mu,\fM, \{T(\myomega)\}_{\myomega\in\Omega}\right)$
is an \emph{operator-valued frame}. This boundedness condition can be expressed by saying that there are $A,B>0$ such that
\[
A\left\Vert x\right\Vert _{\fH}^{2}\le\int_{\Omega}\left\langle T(t)^* T(t)x,x\right\rangle\,d\mu(\myomega)\le B\left\Vert x\right\Vert _{\fH}^{2},
\]
for every $x\in\fM$.
\end{defn}
\noindent Examples of operator-valued frames that are neither continuous frames
nor g-frames arise from the Plancherel theorem for non-compact non-commutative groups, for which examples $A=B$.  (See \cite[Equation~7.46]{gb1995course} for more information.)

If $\left(\mu,\fM, \{T(\myomega)\}_{\myomega\in\Omega}\right)$ is an operator-valued frame
for $\fH$, then the frame operator 
\begin{equation}
S=\int_{\Omega}T(\myomega)^{*}T(\myomega)\,d\mu(\myomega) \nonumber
\end{equation}
 may be defined as before. Here, $S$ is interpreted as an operator
that satisfies
\[
\left\langle Sx,y\right\rangle =\int_{\Omega}\left\langle T(\myomega)^{*}T(\myomega)x,y\right\rangle \,d\mu(\myomega)
\]
for all $x,y\in \fM$. By the boundedness of $T$ and the density of
$\fM$, such an operator exists and is unique. As a result, both the frame algorithm and the procedure of applying
$S^{-1}T^{*}$ to $Tx$ for $x\in \fM$ recover $x$, as before.

A related concept is that of a positive operator-valued measure (POVM).
POVMs represent measurements that occur in open quantum
systems and generalize the concept of a projection-valued measure or von Neumann measurement.  Formally, if $(\Omega,\Sigma)$ is
a measurable space, a POVM is a map from $\Sigma$ to $\fL^+(\fH)$,
the positive semi-definite bounded operators on $\fH$, which is $\sigma$-additive in the
weak operator topology. That is, it is a tuple $(\Omega,\Sigma,M)$, where $M:\Sigma\to\fL^+(\fH)$ is a map such
that
\begin{itemize}
\item $M(\varnothing)=0$ and
\item if $E_{1},E_{2},\dots$ are pairwise disjoint members of $\Sigma$
and $x,y\in\fH$, then 
\[
\left\langle M\left(\cup_{i=1}^{\infty}E_{i}\right)x,y\right\rangle =\sum_{i=1}^{\infty}\left\langle M\left(E_{i}\right)x,y\right\rangle .
\]
\end{itemize}
If, in addition, $M(\Omega)$ is invertible, we say $M$ is a \emph{framed POVM}. We will say that an OVF $\left(\mu, \fM, \{T(\myomega)\}_{\myomega\in\Omega} \right)$ \emph{gives rise to} the 
POVM $M$ if
\begin{equation*} 
\left\langle M(E)x,y\right\rangle =\int_{E}\left\langle T(\myomega)^{*}T(\myomega)x,y\right\rangle \,d\mu(x)
\end{equation*}
for all $x,y\in \fM$ and all $E\in\Sigma$.  



It is easily seen that every g-frame and every continuous frame give rise to a framed POVM.  With a little more work, it can be shown that every OVF gives rise to a framed POVM.  The question of this paper is the converse question: ``Does every framed POVM arise from an OVF?''

\section{Main Result} \label{sec:main-result}

For this section, recall that a closed operator on $\fH$ is a map $A:D(A)\to \fH$ with a closed graph, where $D(A)$ is a dense linear subspace of $\fH$.  In other words, it is a map for which if $x_n \in D(A) \to x\in \fH$ and $Ax_n$ converges to $y\in \fH$, then $x\in D(A)$ and $y=Ax$.

The answer to the question of the last section hinges on the following Radon-Nikodym-type theorem.

\begin{thm*}
Suppose $(\Omega,\Sigma, M)$
is a POVM. Suppose that $\fM$ is a dense linear manifold in $\fH$
and that for each $x\in\fM$, the total variation of the vector measure
$\mu_{x}:E\in\Sigma\mapsto M(E)x$ is sigma-finite. Then there is
a sigma-finite measure $(\Omega,\Sigma,\mu)$ and a positive closed
operator-valued function $Q(t):\fM\rightarrow\fH$, defined for $\mu$-a.e. $t\in\Omega$,
such that 
\begin{equation} 
M(E)x=\int_{E} Q(t)x \,d\mu(t),\label{eq:radon-decomposition-vector}
\end{equation}
weakly,
for all $E\in\Sigma$ and all $x\in\fM$.  
Further, if $(Q_1,\mu_1)$ and $(Q_2,\mu_2)$ are operator-valued functions defined on $\fM$ and sigma-finite measures satisfying (\ref{eq:radon-decomposition-vector}), we have the following operator equality for $(\mu_1+\mu_2)$-a.e. $t$:
\begin{equation} \label{eq:uniqueness}
Q_1(t)\frac{d\mu_1}{d(\mu_1+\mu_2)}(t)= Q_2(t) \frac{d\mu_2}{d(\mu_1+\mu_2)}(t).
\end{equation}
\end{thm*}
\begin{proof}
Let $\mu_{x,y}:\Sigma\to\mathbb{C}$ for $x,y\in\fH$ be the complex
measure defined by $\mu_{x,y}(E)=\left\langle M(E)x,y\right\rangle $.  Let $\mu$ be any sigma-finite measure dominating each $\mu_{x,y}$.  For example, if $\{x_j\}$ is a countable dense subset of the unit ball, then we may define
\begin{equation*}
\mu(E) = \sum_{j\ge 1} \frac{1}{2^j} \mu_{x_j,x_j}(E).
\end{equation*}
Observe that
\begin{align*}
\left|\mu_{x,y}(E)\right| & =\left|\left\langle M(E)x,y\right\rangle \right|\\
 & \le\left\Vert M(E)x\right\Vert \left\Vert y\right\Vert 
\end{align*}
so that
\begin{equation}
\left|\mu_{x,y}\right|(E)\le\left|\mu_{x}\right|(E)\left\Vert y\right\Vert \label{eq:not-exceed-x}
\end{equation}
for all $x\in\fM$ and all $y\in\fH$. We use the term \emph{null set} to mean ``$\mu$-null set.''

For each $x\in\fM$, fix a Radon-Nikodym
derivative $g_{x}$ of $|\mu_{x}|$ with respect to $\mu$. The fact
that $\left|\mu_{x}\right|$ is sigma-finite implies that we may assume
$g_{x}$ is finite everywhere. For each $x,y \in\fM$,
fix a Radon-Nikodym derivative $g_{x,y}$ of $\mu_{x,y}$ with respect
to $|\mu_{x}|$. By (\ref{eq:not-exceed-x}), we may assume $g_{x,y}$
does not exceed $\left\Vert y\right\Vert $ in absolute value. Similarly, we may choose
$g_{x,x}(t)$ to be non-negative for all $t\in\Omega$. Let $f_{x,y}=g_{x,y}g_x$ for $x,y\in\fM$. Then, since both $f_{x,y}$ and $f_{y,x}$ are valid Radon-Nikodym derivatives of $\mu_{x,y}$ with
respect to $\mu$ for $x,y\in\fM$, so is the following:
\[
q(t;x,y):=\begin{cases}f_{x,y}(t), & \text{if } |f_{x,y}(t)| \le |f_{y,x}(t)| \\ 
f_{y,x}(t), & \text{else}.
\end{cases}
\]

Let $\mathbb{F}=\mathbb{Q}+i\mathbb{Q}$, and let $\fM_{0}$ be the
finite $\mathbb{F}$-linear span of a countable dense subset of $\fM$.
Assume $x,y \in \fM_0$ and $t\in \Omega$. By our choice of $q$, we have $|q(t;x,y)|\le g_{y}(t)\left\Vert x\right\Vert $.
Further, $g_{y}(t)$ is finite, so $q(t;\cdot,y)$
is bounded. Let us restrict $t$ to a null-complemented
set $\Omega_{y}\in\Sigma$ such that $q(t;\cdot,y)$ is $\mathbb{F}$-linear.  The functional $q(t;\cdot,y)$  extends to a unique $\mathbb{C}$-linear continuous functional on all of $\fH$.
Thus,
by the Riesz representation theorem, for all $t\in\Omega':=\cap_{y\in\fM_{0}}\Omega_{y}$
there is a vector $z(t;y)\in\fH$ such that
\begin{equation} \label{eq:first-reduction}
q(t;x,y)=\left\langle x,z(t;y)\right\rangle 
\end{equation}
for all $x,y\in\fM_{0}$, and all $t\in\Omega'$.

Let $x\in\fM_{0}$. There is a measurable null-complemented set $\Omega_{x}'\subset\Omega'$
such that $y\in\fM_{0}\mapsto q(t;x,y)$ is $\mathbb{F}$-conjugate-linear
for all $t\in\Omega_{x}'$. Letting $\Omega''=\cap_{x\in\fM_{0}}\Omega_{x}'$,
we may assume this map is $\mathbb{F}$-conjugate-linear for all $x\in\fM_{0}$
and all $t\in\Omega''$. Thus, we have
\[
\left\langle x,z(t;ay+by')\right\rangle =\left\langle x,a\,z(t;y)+b\,z(t;y')\right\rangle 
\]
for all $x,y,y'\in\fM_{0}$, all $a,b\in\mathbb{F}$, and all $t\in\Omega''$.
Letting $x$ range over $\fM_{0}$ this gives $\mathbb{F}$-linearity
of the map $y\in\fM_{0}\mapsto z(t;y)$ for all $t\in\Omega''$. 
Thus, there is an $\mathbb{F}$-linear map $Q_0(t):\fM_0\to \fH$ such that $z(t;y)=Q_0(t)y$ for all $y\in\fM_0$.  Combining this with (\ref{eq:first-reduction}) we get
\[
q(t;x,y) = \left\langle x, Q_0(t)y \right\rangle
\]
for all $x,y\in\fM_0$, and $t\in\Omega''$.

Let $y\in\fM_0$ and $E\in\Sigma$.  We have now shown that 
\begin{equation} \label{eq:desired-result}
\mu_{x,y}(E) = \int_E \left\langle x,Q_0(t)y\right\rangle\, d\mu(t)
\end{equation}
for all $x\in \fM_0$.  It follows that
\begin{equation} \label{eq:total-variation-msr-xy}
|\mu_{x,y}|(E) = \int_E \left| \left\langle x,Q_0(t)y \right\rangle \right|\, d\mu(t)
\end{equation}
for all $x\in \fM_0$.  We now wish to show that (\ref{eq:desired-result}) extends to all $x\in\fH$. In other words, we wish to show that the functional $\phi:\fH \to \mathbb{C}$ defined by
\[
\phi(x)=\int_E \left\langle x,Q_0(t)y\right\rangle\, d\mu(t)
\]
is well-defined and satisfies $\phi(x)=\mu_{x,y}(E)$ for all $x\in\fH$.  For this, suppose $x\in\fH$ and $x_n$ is a sequence in $\fM_0$ converging to $x$. To show $\phi$ is well-defined, we first note that by Fatou's lemma we have
\begin{align} 
& \left(\int_E\left|\left\langle x,Q_0(t)y\right\rangle\right|\, d\mu(t)\right)^2 \nonumber \\
& = \left(\int_E\lim_n \left|\left\langle x_n,Q_0(t)y\right\rangle\right|\, d\mu(t)\right)^2 \nonumber \\
& \le \liminf_n \left(\int_E \left|\left\langle x_n,Q_0(t)y\right\rangle\right|\, d\mu(t)\right)^2 \nonumber 
\end{align}
By Cauchy-Schwarz, this is bounded by
\begin{align}
& \le \liminf_n \int_E \left\langle x_n,Q_0(t)x_n\right\rangle\, d\mu(t) \int_E \left\langle y,Q_0(t)y\right\rangle\, d\mu(t). \nonumber
\end{align}
Further, by (\ref{eq:desired-result}), this is equal to
\begin{align}
& \le \liminf_n \left\langle M(E)x_n,x_n\right\rangle \left\langle M(E)y,y\right\rangle \nonumber \\
& = \left\langle M(E)x,x\right\rangle \left\langle M(E)y,y\right\rangle \label{eq:continuity-xy}
\end{align}
This means that $\phi$ is well-defined.  Futher, (\ref{eq:continuity-xy}) means that $\phi$ is continuous.
Since $x\in\fH\mapsto \mu_{x,y}(E)$ is also continuous and restricts to $\phi$ on $\fM_0$, we have $\phi(x)=\mu_{x,y}(E)$, as desired.

We will now show that there is a closed, positive operator-valued function $Q(t)$ with domain $\fM$ such that
\begin{equation} \label{eq:final-desired-result}
\mu_{x,y}(E) = \int_E \left\langle Q(t)x,y\right\rangle\, d\mu(t),
\end{equation}
for all $x\in\fM$, all $y\in\fM_0$, and all $E\in\Sigma$. For this, let $x\in\fM$, $y\in\fM_0$, and $E\in\Sigma$.  By the conclusion of the last paragraph and the definition of total-variation measure, we have:
\[
 \int_E \left| \left\langle x, Q_0(t)y\right\rangle\right| \, d\mu(t)= |\mu_{x,y}|(E).
\]
Using (\ref{eq:not-exceed-x}), the right hand side is bounded by $|\mu_x|(E)\left\Vert y\right\Vert$.  But since $E$ was arbitrary, this means that
\[
\left| \left\langle x, Q_0(t)y\right\rangle\right| \le g_x(t) \left\Vert y\right\Vert,
\]
for every $t$ in a null-complemented set $\Omega'''$. Fix $t\in\Omega'''$. By finiteness of $g_x(t)$, the above display equation means that if $y_n$ is a sequence in $\fM_0$ converging to $y$ we have
\[
\left\langle x, Q_0(t) y_n\right\rangle - \left\langle x, Q_0(t)y \right\rangle \to 0.
\]
Thus, by the Riesz representation theorem, there is a densely defined operator $Q(t)$ with domain $\fM$ such that 
\[
\left\langle x, Q_0(t)y \right\rangle = \left\langle Q(t)x, y\right\rangle. 
\]
 Further, $Q(t)$ is closed  by \cite[5.1.5]{pedersen2012analysis} and positive since $Q_0(t)$ is positive.
 
Let $x\in\fM$ and $E\in\Sigma$. We will now argue that (\ref{eq:final-desired-result}) extends to all $y\in\fH$.  For this we define
$\psi:\fH\to\mathbb{C}$ by
 \[
 \psi(y) = \int_E \left\langle Q(t)x, y\right\rangle\, d\mu(t).
 \]
 It suffices to show that $\psi$ is well-defined and agrees with $\mu_{x,y}(E)$ for all $y\in\fH$.  But this follows from the extension argument just applied to $\phi:\fM_0\to\mathbb{C}$ in the paragraph before last.
 This concludes the proof of (\ref{eq:radon-decomposition-vector}).

For (\ref{eq:uniqueness}), suppose that
\begin{equation*}
\int_E \left\langle Q_1(t)x,y \right\rangle\, d\mu_1(t) = \int_E \left\langle Q_2(t)x,y\right\rangle\, d\mu_2(t),
\end{equation*}
for all $E\in \Sigma$ and all $x\in\fM$ and $y\in\fH$.
Then for $(\mu_1+\mu_2)$-a.e. $t$ and all $x\in \fM$ and $y\in\fH$, we have
\begin{equation*}
\left\langle Q_1(t)x,y \right\rangle \frac{d\mu_1}{d(\mu_1+\mu_2)}(t) = \left\langle Q_2(t)x,y\right\rangle \frac{d\mu_2}{d(\mu_1+\mu_2)}(t) .
\end{equation*}
But this means
\begin{equation*}
Q_1(t)\frac{d\mu_1}{d(\mu_1+\mu_2)}(t) = Q_2(t)\frac{d\mu_2}{d(\mu_1+\mu_2)}(t),\end{equation*}
as desired.

\end{proof}

The above yields the sigma-finite case (\cite[Theorem~3.3.2]{robinson2014operator}) as an immediate corollary.  Here, we say that $E\mapsto M(E)$ is sigma-finite if its total variation with respect to the operator norm is sigma-finite.

\begin{cor*}
Suppose $(\Omega, \Sigma, M)$ is a sigma-finite POVM. Then $\mu_x$ is sigma-finite for all $x\in \fH$ and $Q(t)$ is bounded.
\end{cor*}
\begin{proof}
We may replace $\mu$ in the previous proof by the total variation measure of $M$ since it dominates $\mu_{x,y}$ for all $x,y\in\fH$. 
 Further, $\mu(E)\left\Vert x \right\Vert$ dominates the total variation measure of $E\mapsto M(E)x$, so the latter is sigma-finite for all $x\in\fH$.
By the definition of the total-variation measure of $\mu_{x,y}$ and $M$, we have
\begin{align*}
\int_E \left|\left\langle Q(t)x,y\right\rangle\right|\, d\mu(t) & \le \mu(E) \left\Vert x \right\Vert \left\Vert y \right\Vert 
\end{align*}
for all $E\in\Sigma$ and all $x,y\in\fH$.  Thus, the Radon-Nikodym theorem tells us that for $\mu$-a.e. $t$ and all $x,y\in \fH$,
\[
\left|\left\langle Q(t)x,y \right\rangle\right| \le \left\Vert x \right\Vert \left\Vert y \right\Vert,
\]
which means that $Q(t)$ is bounded for $\mu$-a.e. $t$.  
\end{proof}

It follows immediately from the Theorem that any framed POVM $M$ satisfying the given sigma-finiteness condition arises from an OVF, as desired.  Indeed, letting $\mu$, $Q(t)$, and $\fM$ be as in the Theorem, such an OVF is $(\mu, \fM, \{T(t)\}_{t\in\Omega})$, where $T(t)=Q(t)^{1/2}$.  Further, the uniqueness condition of the Theorem shows that if $(\mu_1, \fM, \{T_1(t)\}_{t\in\Omega})$ and $(\mu_2,\fM, \{T_2(t)\}_{t\in\Omega})$ are any two OVFs giving rise to $M$, then they are essentially the same in the sense that 
$$
T_1(t)^*T_1(t)\frac{d\mu_1}{d(\mu_1+\mu_2)}(t)=T_2(t)^*T_2(t)\frac{d\mu_2}{d(\mu_1+\mu_2)}(t)
$$
for $(\mu_1+\mu_2)$-a.e. $t$. This concludes our argument.


\bibliographystyle{plain}
\bibliography{thesisbib}

\end{document}